\documentclass{article}

\usepackage{fullpage}
\usepackage[utf8]{inputenc}
\usepackage[T2A]{fontenc}
\usepackage{subfigure}

\usepackage{xcolor}
\definecolor{lgreen}{rgb}{0.0, 0.48, 0.0}
\definecolor{lpurple}{rgb}{0.48, 0.0, 0.48}

\usepackage{amsfonts,amssymb,amsmath,dsfont}
\usepackage{amsthm}
\usepackage{enumitem}
  \setlist{nosep} 
\usepackage{xspace}
\usepackage{graphicx}
\usepackage{bm}
\usepackage{comment}

\usepackage[breaklinks]{hyperref}
\usepackage{url}

\usepackage{tikz}
\usepackage{aliascnt, mathtools}

\usepackage{soul} 

\begin{document}


\title{Symbolic method and directed graph enumeration\thanks{\'Elie de Panafieu is a member of Lincs \texttt{www.lincs.fr} and Sergey Dovgal is supported by the French ANR project MetACOnc, ANR-15-CE40-0014.
The order of the authors is alphabetical.}
}


\author{%
\'Elie de Panafieu\\ {\normalsize Bell Labs France, Nokia}
\and
Sergey Dovgal\\ {\normalsize LIPN, Institut Galilée Université Paris 13}
}

%
%







\maketitle

\begin{abstract}
We introduce the arrow product,
a systematic generating function technique
for directed graph enumeration.
It provides short proofs for previous results of Gessel
on the number of directed acyclic graphs
and of Liskovets, Robinson and Wright
on the number of strongly connected directed graphs.
We also recover Robinson's enumerative results
on directed graphs where all strongly connected components
belong to a given family.

\noindent \textbf{keywords.} directed graph, digraph, analytic combinatorics, generating functions
\end{abstract}

\theoremstyle{definition}
\newtheorem{theorem}{Theorem}[section]

\newaliascnt{corollary}{theorem}
\newtheorem{corollary}[corollary]{Corollary}
\aliascntresetthe{corollary}
\providecommand*{\corollaryautorefname}{Corollary}

\newaliascnt{lemma}{theorem}
\newtheorem{lemma}[lemma]{Lemma}
\aliascntresetthe{lemma}
\providecommand*{\lemmaautorefname}{Lemma}

\newaliascnt{definition}{theorem}
\newtheorem{definition}[definition]{Definition}
\aliascntresetthe{definition}
\providecommand*{\definitionautorefname}{Definition}

\newaliascnt{proposition}{theorem}
\newtheorem{proposition}[proposition]{Proposition}
\aliascntresetthe{proposition}
\providecommand*{\propositionautorefname}{Proposition}

\newaliascnt{remark}{theorem}
\newtheorem{remark}[remark]{Remark}
\aliascntresetthe{remark}
\providecommand*{\remarkautorefname}{Remark}

\renewcommand{\geq}{\geqslant}
\renewcommand{\vec}{\boldsymbol}


\newcommand{\proba}{\mathds{P}}
\newcommand{\bigO}{\mathcal{O}}
\newcommand{\exactbigO}{\Theta}
\newcommand{\smallo}{o}
\newcommand{\integers}{\mathds{Z}}
\newcommand{\reals}{\mathds{R}}
\newcommand{\naturals}{\integers_{\geq 0}}
\newcommand{\complex}{\mathds{C}}
\newcommand{\permutations}{\mathfrack{S}}
\newcommand{\indic}{\mathds{1}}
\newcommand{\mA}{\mathcal{A}}
\newcommand{\mB}{\mathcal{B}}
\newcommand{\mC}{\mathcal{C}}
\newcommand{\mD}{\mathcal{D}}
\newcommand{\mF}{\mathcal{F}}
\newcommand{\mH}{\mathcal{H}}
\newcommand{\mJ}{\mathcal{J}}
\newcommand{\mK}{\mathcal{K}}
\newcommand{\mN}{\mathcal{N}}
\newcommand{\mM}{\mathcal{M}}
\newcommand{\mL}{\mathcal{L}}
\newcommand{\mT}{\mathcal{T}}
\newcommand*{\eg}{\textit{e.g.}\@\xspace}
\newcommand*{\ie}{\textit{i.e.}\@\xspace}
\newcommand*{\aka}{\textit{a.k.a.}\@\xspace}
\newcommand*{\resp}{resp.\@\xspace}
\newcommand{\Real}{\exponential{Re}}


\newcommand{\exponential}[1]{\mathrm{#1}}

\newcommand{\eA}{\exponential{A}}
\newcommand{\eB}{\exponential{B}}
\newcommand{\eT}{\exponential{T}}
\newcommand{\egraph}{\exponential{G}}
\newcommand{\edigraph}{\exponential{D}}
\newcommand{\eset}{\exponential{Set}}
\newcommand{\econnected}{\exponential{C}}
\newcommand{\einv}{\exponential{Inv}}
\newcommand{\estrongly}{\exponential{SCC}}
\newcommand{\eW}{\exponential{W}}


\newcommand{\graphic}[1]{\mathbf{#1}}

\newcommand{\gA}{\graphic{A}}
\newcommand{\gB}{\graphic{B}}
\newcommand{\ggraph}{\graphic{G}}
\newcommand{\gdigraph}{\graphic{D}}
\newcommand{\gset}{\graphic{Set}}
\newcommand{\ginitially}{\graphic{IC}}
\newcommand{\gdag}{\graphic{DAG}}
\newcommand{\gW}{\graphic{W}}

\newcommand{\hadamard}{\odot}


\newcommand{\family}[1]{\mathcal{#1}}

\newcommand{\fA}{\family{A}}
\newcommand{\fB}{\family{B}}
\newcommand{\fC}{\family{C}}
\newcommand{\fD}{\family{D}}
\newcommand{\fW}{\family{W}}
\newcommand{\fgraph}{\family{G}}
\newcommand{\fdigraph}{\family{D}}
\newcommand{\fset}{\family{Set}}
\newcommand{\finitially}{\family{IC}}
\newcommand{\fdag}{\family{DAG}}

\newcommand{\ic}{\mathit{ic}}

\definecolor{bblue}{rgb}{0.2, 0.4, 0.8}
\definecolor{bgreen}{rgb}{0.2, 0.6, 0.4}
\definecolor{bred}{rgb}{0.8, 0.4, 0.2}
\definecolor{bviolet}{rgb}{0.7, 0.2, 0.7}
\definecolor{blackred}{rgb}{0.6, 0.3, 0.3}
\definecolor{blackblue}{rgb}{0.3, 0.3, 0.6}

\usetikzlibrary{fit,arrows,trees,shapes,shapes.geometric,calc,matrix,
decorations.markings}
\tikzset{
  treenode/.style = {align=center, inner sep=0pt, text centered,
    font=\sffamily},
  arnBleuPetit/.style = {treenode, circle, bblue, draw=bblue,
    fill=bblue!10,
    minimum width=0.8em, minimum height=0.5em
  },
  arnRougePetit/.style = {treenode, circle, bred, draw=bred,
    fill=bred!40,
    minimum width=0.8em, minimum height=0.5em
  },
  arnBleuGrande/.style = {treenode, circle, bblue, draw=bblue,
    text width=1.5em, very thick,
    fill=bblue!10},
  arnVioletGrande/.style = {treenode, circle, bviolet, draw=bviolet,
    text width=1.5em, very thick,
    fill=bblue!10},
}


      \section{Introduction}


\sloppy
The enumeration of two important digraph families,
the \emph{Directed Acyclic Graphs} (DAGs)
and the strongly connected digraphs,
has been successfully approached at least since 1969.
Apparently, it was Liskovets~\cite{liskovets1969,liskovets1970number}
who first deduced a recurrence for the number
of strongly connected digraphs and also introduced and
studied the concept of initially connected digraph,
a helpful tool for their enumeration.
Subsequently, Wright~\cite{wright1971number} derived
a simpler recurrence for strongly connected digraphs
and Liskovets~\cite{liskovets1973} extended his techniques
to the unlabeled case.
Stanley counted labeled DAGs in~\cite{stanley1973acyclic}, and Robinson,
in his papers~\cite{robinson1973,robinson1977counting},
counted labeled and unlabeled DAGs with a given number of sources,
obtained exact and asymptotic results on counting of different families of
digraphs, including strongly connected digraphs,
which was the culmination of a series of publications he started in 1970
independently of Stanley.
In the unlabeled case, his approach is very much related
to the Species Theory~\cite{bergeron1998combinatorial}
which systematises the usage of cycle index series.
Robinson also announced~\cite{robinson1977strong}
a simple combinatorial explanation for the generating function
of strongly connected digraphs
in terms of the cycle index function.
%
Publications on the exact enumeration of digraphs slowed down,
until Gessel~\cite{GESSEL1995257} and Robinson~\cite{RobinsonGeneral}
in 1995, returned to the problem with the approach of \emph{graphic generating
functions} or \emph{special generating functions}\footnote{In fact, Robinson's
paper~\cite{robinson1973} from 1973
entitled ``Counting labeled acyclic digraphs''
already contains the notion of a special generating function and the method
analogous to the method that we describe in the current paper. Robinson also
obtains a simple expression for the generating function
of strongly connected digraphs, similar to~\autoref{th:strongly_connected}.
The authors have discovered Robinson's
papers~\cite{robinson1973,RobinsonGeneral} after the body of this work was finished and
accepted for a publication. We are solving here the same problem with a
similar method. We chose to maintain the publication, as we felt those results are of interest for the scientific community, and did not yet received the diffusion they deserve. Our further goal is to integrate these exact
enumeration methods into an asymptotic framework in the future.}.
It allowed them to enumerate
DAGs by marking sources and sinks~\cite{G96} and digraphs by marking source-like
and sink-like components.


The symbolic method~\cite{bergeron1998combinatorial,FS09}
is a dictionary that translates combinatorial operations
into generating function relations.
In particular, it allows to manipulate the generating functions
directly, avoiding working at the coefficient level.
Our contribution is twofold.
Firstly, we describe a new operation,
the \textit{arrow product} (\autoref{def:arrow_product}),
which enriches the symbolic method.
Secondly, we propose simple proofs, similar to those of \cite{RobinsonGeneral},
for the generating functions of directed acyclic digraphs (DAGs),
strongly connected graphs (SCCs),
and digraphs where all SCCs belong to a given family.
Some variants are presented as well.

Similar techniques enabled precise description of simple graphs phase transition (see e.g.~\cite{JKLP93}),
so the techniques developed here might enable the study
of digraphs phase transition~\cite{Goldschmidt, luczak2009critical}.

In this paper, we consider directed graphs (digraphs)
with labeled vertices, without loops or multiple edges.
Two vertices $u$, $v$ can be simultaneously
linked by both edges $u \to v$ and $v \to u$.
We also consider simple graphs
which are undirected graphs with neither multiple edges nor loops.

    \section{The symbolic approach}

  \subsection{Definitions}

Consider a sequence \( (a_n(w))_{n = 0}^\infty \). Define the \textit{exponential generating function} (EGF) and the \textit{graphic generating function} (GGF)
(introduced in~\cite{GESSEL1995257}) of the sequence \( (a_n(w))_{n = 0}^\infty
\) as
\[
  \eA(z, w) := \sum_{n \geq 0} a_n(w) \frac{z^n}{n!}
  \quad \text{ and } \quad
  \gA(z, w) := \sum_{n \geq 0} \frac{a_n(w)}{(1+w)^{\binom{n}{2}}} \frac{z^n}{n!}.
\]
To distinguish EGF from GGF, the latter are written in bold characters.
The \emph{special generating functions} of \cite{RobinsonGeneral} correspond to GGFs with $w=1$.
The $n$th coefficient of a series $A(z)$ with respect to the variable $z$ is denoted by $[z^n] A(z)$, so
\(
  A(z) = \sum_{n \geq 0} ([x^n] A(x)) z^n
\).

The \emph{exponential Hadamard product} of two series
\(
A(z) = \sum_{n \geq 0} a_n \tfrac{z^n}{n!}
\)
and
\(
B(z) = \sum_{n \geq 0} b_n \tfrac{z^n}{n!}
\)
is denoted by and defined as
\[
    A(z) \hadamard B(z) =
\Big(
\sum_{n \geq 0} a_n \dfrac{z^n}{n!}
\Big)
\hadamard
\Big(
\sum_{n \geq 0} b_n \dfrac{z^n}{n!}
\Big)
:=
\sum_{n \geq 0} a_n b_n \dfrac{z^n}{n!}
.
\]
All Hadamard products are taken with respect to the variable $z$.
The Hadamard product can be used to convert between EGF and GGF
(see~\autoref{th:translation}).
The exponential Hadamard product should not be confused
with the ordinary Hadamard product
$\sum_n ([z^n] A(z)) ([z^n] B(z)) z^n$.

If \( \fA \) is a certain family of digraphs or graphs, we can associate to it a
sequence of series \( (a_n(w))_{n = 0}^\infty \), such that $[w^m] a_n(w)$ is
equal to the number of elements in $\fA$ with $n$ vertices and $m$ directed edges.
Consequently, we can associate both EGF and GGF to
the same family of digraphs or graphs.

An advantage of the symbolic method is its ability to keep track of a
collection of \emph{parameters} in combinatorial objects. The two default
parameters are the numbers of vertices and edges, and the arguments \( z \) and
\( w \) of a generating function \( F(z, w) \) correspond to these parameters.
As a generalization, we consider multivariate generating functions
\[
    \eA(z, w, \vec u)
    :=
    \sum_{n, \vec p}
        a_{n, \vec p}(w) \vec u^{\vec p}
    \frac{z^n}{n!}
    \quad \text{ and } \quad
    \gA(z, w, \vec u)
    :=
    \sum_{n, \vec p}
    \frac{a_{n,\vec p}(w) \vec u^{\vec p}}{(1+w)^{\binom{n}{2}}} \frac{z^n}{n!},
\]
where
\( \vec u = (u_1, \cdots, u_d) \) is the vector of variables,
\( \vec p = (p_1, \cdots, p_d) \) denotes a vector of parameters,
and the notation \( \vec u^{\vec p} := \prod_k u_k^{p_k} \)
is used. We say that the variable \( u_k \) \emph{marks} its corresponding
parameter $p_k$, see~\cite{FS09}.

  \subsection{Combinatorial operations}

The next proposition recalls classic operations on EGFs (see~\cite{FS09}),
which extend naturally to GGFs.

\begin{proposition} \label{th:symbolic_method}
Consider two digraph (or graph) families $\fA$ and $\fB$.
The EGF and GGF of the disjoint union of $\fA$ and $\fB$ are $\eA(z,w) + \eB(z,w)$ and $\gA(z,w) + \gB(z,w)$.
The EGF and GGF of the digraphs from $\fA$ where one vertex is distinguished
are $z \partial_z \eA(z)$ and $z \partial_z \gA(z,w)$.
The EGF of sets of digraphs from $\fA$ is $e^{\eA(z,w)}$.
The EGF of pairs of digraphs $(a, b)$ with $a \in \fA$ and $b \in \fB$ (relabeled so that the vertex labels of $a$ and $b$ are disjoint, see~\cite{FS09}) is $\eA(z,w) \eB(z,w)$.
If a variable \( u \) marks the number of specific items in the EGF \( \eA(z, w, u)
\) or the GGF \( \gA(z, w, u) \) of the family \( \fA \), then the EGF and GGF for
the objects \( a \in \fA \) which have a distinguished subset of these specific items
are \( \eA(z, w, u+1) \) and \( \gA(z, w, u+1) \). Replacing \( u \mapsto u-1 \)
corresponds to an inclusion-exclusion process.
\end{proposition}

The next definition and proposition translate the combinatorial interpretation of the product of GGFs, already mentioned by \cite{RobinsonGeneral}, into the symbolic method framework.
Gessel also used it implicitely in several proofs (\eg \cite{G96}) at coefficient level, but did not express it at the generating function level. However, a combinatorial interpretation of the exponential of GGFs can be found in~\cite{GESSEL1995257,gessel1996tutte}.

\newcommand{\fromto}[2]{
    \path (#1) edge [black,thick,
    decoration={markings,mark=at position 1 with
    {\arrow[ultra thick,blackblue, rotate=0]{>}}}, postaction={decorate}
    ] node {} (#2);
}
\newcommand{\aprod}[2]{
    \path (#1) edge [blackred,thick,
    decoration={markings,mark=at position 1 with
    {\arrow[ultra thick,blackred, rotate=0]{>}}}, postaction={decorate}
    ] node {} (#2);
}

\begin{figure}[!htb]
   \begin{minipage}[t]{0.31\textwidth}
   \centering
\scalebox{0.8}{
\begin{tikzpicture}[>=stealth',thick, scale = 0.8]
\draw
node[arnBleuPetit](a) at ( 0, 0)  { }
node[arnBleuPetit](s) at ( 1,-1)  { }
node[arnBleuPetit](d) at ( 2, 0)  { }
node[arnBleuPetit](f) at ( 1, 1)  { }
;
\fromto{a}{d};
\fromto{a}{f};
\fromto{s}{a};
\fromto{d}{s};
\fromto{f}{d};
\draw
node[arnBleuPetit](q)  at (3.8, 0)  { }
node[arnBleuPetit](w)  at (3,  -1)  { }
node[arnBleuPetit](e)  at (4.8,-1)  { }
node[arnBleuPetit](r)  at (4.8, 0)  { }
node[arnBleuPetit](t)  at (3.8, 1)  { }
;
\fromto{q}{w};
\fromto{q}{r};
\fromto{e}{r};
\fromto{r}{t};
\fromto{t}{q};
\aprod{d}{t};
\aprod{d}{q};
\aprod{s}{w};
\node[rectangle,dashed,draw,fit=(a)(s)(d)(f), very thick,
      rounded corners=5mm,inner sep= 5pt, bgreen] {};
\node[rectangle,dashed,draw,fit=(q)(w)(e)(r)(t), very thick,
      rounded corners=5mm,inner sep= 5pt, bgreen] {};
\end{tikzpicture}
}
     \caption{The arrow product. The vertex labels have been omitted}\label{fig:arrow:product}
   \end{minipage}\hfill
   \begin{minipage}[t]{0.3\textwidth}
     \centering
\scalebox{0.9}{
\begin{tikzpicture}[>=stealth',thick, scale = 0.8]
\draw
node[arnBleuPetit](a) at ( 1,.8)  {}
node[arnBleuPetit](s) at ( 1,-1)  {}
node[arnBleuPetit](d) at (1.8,0)  {}
;
\draw
node[arnBleuPetit](q)  at ( 4, 0)  { }
node[arnBleuPetit](w)  at ( 3,-1)  { }
node[arnBleuPetit](e)  at ( 5,-1)  {}
node[arnBleuPetit](r)  at ( 5, 0)  { }
node[arnBleuPetit](t)  at ( 4, 1)  { }
;
\fromto{q}{w};
\fromto{q}{r};
\fromto{e}{r};
\fromto{r}{t};
\fromto{q}{t};
\aprod{d}{t};
\aprod{d}{q};
\aprod{s}{w};
\node[rectangle,dashed,draw,fit=(a)(s)(d), very thick,
      rounded corners=5mm,inner sep= 5pt, bgreen] {};
      \node[rectangle,dashed,draw,fit=(q)(w)(e)(r)(t), very thick,
      rounded corners=5mm,inner sep= 5pt, bgreen] {};
\end{tikzpicture}
}
     \caption{Symbolic method for DAG}\label{fig:dag:construction}
   \end{minipage}\hfill
   \begin{minipage}[t]{0.36\textwidth}
     \centering
\begin{tikzpicture}[>=stealth',thick, scale = 0.9]
\draw
node[arnBleuPetit,scale=0.9] (aA) at (1.2, 1.1)  {}
node[arnBleuPetit,scale=0.9] (aS) at ( .8,.4)  {}
node[arnBleuPetit,scale=0.9] (aD) at (1.6,.4)  {}

node[arnBleuPetit,scale=0.9] (sA)  at (3.0, .2) {}
node[arnBleuPetit,scale=0.9] (sS)  at (3.0,-.6) {}
node[arnBleuPetit,scale=0.9] (sD)  at (2.4,-.2) {}

node[arnBleuPetit,scale=0.9](dA) at ( 1.3, -1.2)   {}
node[arnBleuPetit,scale=0.9](dS) at ( 1.8,  -.7)   {}
node[arnBleuPetit,scale=0.9](dD) at ( 1.3,  -.2)   {}
node[arnBleuPetit,scale=0.9](dF) at (  .8,  -.7)   {}
;
\draw
node[arnBleuPetit,scale=0.9](q)  at ( 4, 0)  { }
node[arnBleuPetit,scale=0.9](w)  at ( 4,-1)  { }
node[arnBleuPetit,scale=0.9](e)  at ( 5,-1)  { }
node[arnBleuPetit,scale=0.9](r)  at ( 5, 0)  { }
node[arnBleuPetit,scale=0.9](t)  at ( 4, 1)  { }
;
\fromto{q}{w};
\fromto{q}{r};
\fromto{e}{r};
\fromto{r}{t};
\fromto{q}{t};
\fromto{aA}{aS};
\fromto{aS}{aD};
\fromto{aD}{aA};
\fromto{sA}{sS};
\fromto{sS}{sD};
\fromto{sD}{sA};
\fromto{dA}{dS};
\fromto{dS}{dD};
\fromto{dD}{dF};
\fromto{dF}{dA};
\aprod{sA}{t};
\aprod{sA}{q};
\aprod{dA}{w};
\aprod{aA}{t};
\aprod{sS}{q};
\node[rectangle,dashed,draw,fit=(aA)(sA)(sS)(sD)(dA)(dS)(dD)(dF),
      rounded corners = 5mm,inner sep = 5pt, bgreen, very thick] {};
\node[rectangle,dashed,draw,fit=(q)(w)(e)(r)(t),
      rounded corners = 5mm,inner sep = 5pt, bgreen, very thick] {};
\node[rectangle,dashed,draw,fit=(dA)(dS)(dD)(dF),
      rounded corners = 3mm,inner sep = 2pt, bviolet, very thick] {};
\node[rectangle,dashed,draw,fit=(sA)(sS)(sD),
      rounded corners = 3mm,inner sep = 2pt, bviolet, very thick] {};
\node[rectangle,dashed,draw,fit=(aA)(aS)(aD),
      rounded corners = 3mm,inner sep = 2pt, bviolet, very thick] {};
\end{tikzpicture}
     \caption{Marking a subset of source-like SCC}\label{fig:scc:construction}
   \end{minipage}
\end{figure}

\begin{definition} \label{def:arrow_product}
We define the \emph{arrow product} of $\fA$ and $\fB$ as
the family \( \fC \) of pairs $(a, b)$, with $a \in \fA$, $b \in \fB$
(relabeled so that $a$ and $b$ have disjoint labels),
where an arbitrary number of edges oriented from vertices of $a$ to vertices of $b$ are added (see \autoref{fig:arrow:product}).
\end{definition}

\begin{proposition}
The GGF of the arrow product of the families $\fA$ and $\fB$
is equal to $\gA(z,w) \gB(z,w)$.
\end{proposition}

\begin{proof}
Consider two digraph families $\fA$ and $\fB$,
with associated sequences $(a_n(w))$, $(b_n(w))$.
Then the sequence associated to the GGF $\gA(z,w) \gB(z,w)$ is
\begin{multline*}
    c_n(w)
    = (1+w)^{\binom{n}{2}} n! [z^n]
        \bigg( \sum_k \frac{a_k(w)}{(1+w)^{\binom{k}{2}}} \frac{z^k}{k!} \bigg)
        \bigg( \sum_{\ell} \frac{b_{\ell}(w)}{(1+w)^{\binom{\ell}{2}}} \frac{z^{\ell}}{\ell!} \bigg)
    = \binom{n}{k} \sum_{k+\ell = n} (1+w)^{k \ell} a_k(w) b_{\ell}(w).
\end{multline*}
This series has the following combinatorial interpretation:
it is the generating function (the variable $w$ marks the edges)
of digraphs with $n$ vertices, obtained by
\begin{itemize}
\item
choosing digraphs $a$ of size $k$ in $\fA$,
$b$ of size $\ell$ in $\fB$, such that $k+\ell = n$,
\item
choosing a subset of $\{1, \ldots, n\}$ for the labels of $a$
(and $b$ receives the complementary set for its labels),
\item
for any vertices $u$ in $a$, $v$ in $b$,
the oriented edge $(u,v)$ is or not added.
\end{itemize}
Hence, $(c_n(w))$ is the sequence associated to the arrow product of $\fA$ and
$\fB$.
\end{proof}

    \section{Generating functions from the symbolic method}

We start by defining the building bricks for the symbolic method of directed
graphs.

\begin{proposition} \label{th:first_egf_ggf}
The EGF of all graphs $\egraph(z,w)$,
GGF of all digraphs $\gdigraph(z,w)$,
and GGF of sets $\gset(z,w)$ (labeled graphs that contain no edge) are
\[
  \egraph(z,w) = \gdigraph(z,w) = \sum_{n \geq 0} (1+w)^{\binom{n}{2}} \frac{z^n}{n!}
  \quad \text{ and } \quad
  \gset(z,w) = \sum_{n \geq 0} \frac{1}{(1+w)^{\binom{n}{2}}} \frac{z^n}{n!}.
\]
\end{proposition}

\begin{proof}
Consider a graph with $n$ vertices.
Each unordered pair of distinct vertices
is either linked by an edge, or not.
Thus, the sequence of series associated to the family of graphs and its EGF are
\[
  g_n(w) = (1+w)^{\binom{n}{2}}, \quad
  \egraph(z,w)
  = \sum_{n \geq 0} g_n(w) \frac{z^n}{n!}
  = \sum_{n \geq 0} (1+w)^{\binom{n}{2}} \frac{z^n}{n!}.
\]
%
%
In a digraph with $n$ vertices, each ordered pair of distinct vertices
is either linked by an oriented edge, or not.
So the sequence of series associated to the family of digraphs and its GGF are
\[
  d_n(w) = (1+w)^{n (n-1)}, \quad
  \gdigraph(z,w)
  = \sum_{n \geq 0} \frac{d_n(w)}{(1+w)^{\binom{n}{2}}} \frac{z^n}{n!}
  = \sum_{n \geq 0} (1+w)^{\binom{n}{2}} \frac{z^n}{n!}.
\]
%
There is exactly one labeled graph without any edges,
so the sequence of series associated to the set family and its GGF are
\[
  \mathit{set}_n(w) = 1, \quad
    \gset(z,w) = \sum_{n \geq 0} \frac{1}{(1+w)^{\binom{n}{2}}} \frac{z^n}{n!}.
\]
\end{proof}

\begin{corollary} \label{th:translation}
The EGF and GGF of a family $\fA$ are linked by the relations
\[
  \eA(z,w) = \egraph(z) \hadamard \gA(z,w)
  \quad \text{ and } \quad
  \gA(z) = \gset(z,w) \hadamard \eA(z,w).
\]
\end{corollary}

\begin{proof}
Consider a family $\fA$ with sequence of series $(a_n(w))$.
By definition of the EGF, GGF and exponential Hadamard product, we have
\[
  \egraph(z) \hadamard \gA(z)
  = \bigg( \sum_n (1+w)^{\binom{n}{2}} \frac{z^n}{n!} \bigg) \hadamard \sum_n \frac{a_n(w)}{(1+w)^{\binom{n}{2}}} \frac{z^n}{n!}
  = \sum_n a_n(w) \frac{z^n}{n!}
  = \eA(z),
\]
and similarly
\[
  \gset(z) \hadamard \eA(z)
  = \bigg( \sum_n \frac{1}{(1+w)^{\binom{n}{2}}} \frac{z^n}{n!} \bigg) \hadamard \sum_n a_n(w) \frac{z^n}{n!}
  = \sum_n \frac{a_n(w)}{(1+w)^{\binom{n}{2}}} \frac{z^n}{n!}
  = \gA(z).
\]
\end{proof}

  \subsection{Generating functions of various digraph families}

The next proposition comes from~\cite{G96, robinson1977counting, stanley1973acyclic}. We present a proof relying on the arrow product.

\begin{proposition} \label{th:dag}
The GGF of directed acyclic graphs (DAGs)
with an additional variable $u$ marking the sources
(\ie there are no oriented edge pointing to those vertices) is
\[
  \gdag(z,w,u) = \frac{\gset((u-1) z,w)}{\gset(- z,w)}.
\]
\end{proposition}

\begin{proof}
The GGF of DAGs where each source is either marked,
or left unmarked by the variable $u$, is $\gdag(z,w,u+1)$
(see~\autoref{th:symbolic_method}).
Such a DAG is decomposed as the arrow product
of a set (the marked sources) with a digraph
(\autoref{fig:dag:construction}), so
\[
  \gdag(z,w,u+1) = \gset(z u, w) \gdag(z,w).
\]
Observe that $\gdag(z,w,0)$ is the GGF of DAGs without any source.
The only DAG satisfying this property is the empty DAG, so
$\gdag(z,w,0) = 1$.
Taking $u = -1$ gives $1 = \gset(-z,w) \gdag(z,w)$,
so $\gdag(z,w) = 1 / \gset(-z,w)$.
Replacing $u$ with $u-1$ gives $\gdag(z,w,u) = \gset((u-1) z,w) / \gset(-z,w)$.
This second proof also illustrates the translation
into the generating function world
of the inclusion-exclusion principle.
\end{proof}


Let us recall that the \emph{condensation} of a digraph
is the directed acyclic graph (DAG) obtained from it
by contracting each strongly connected component (SCC)
to a vertex.
The SCCs of the digraph corresponding to sources of the condensation
are called \emph{source-like SCCs}.
The proof from \autoref{th:dag}
for expressing the generating function of DAGs with marked sources
is now extended to digraphs with marked source-like components
and SCCs belonging to a given family (similar proof published by~\cite{RobinsonGeneral}).

\begin{theorem} \label{th:source_like}
Consider a nonempty family $\mA$ of SCCs
(the empty digraph is not strongly connected by convention,
so it cannot belong to $\mA$).
The GGF of digraphs where all SCCs belong to $\mA$ is equal to
\[
  \gdigraph_{\mA}(z,w) = \frac{1}{\gset(w,z) \hadamard e^{-\eA(z,w)}}.
\]
The GGF of the same digraph family where an additional variable $u$
marks the source-like components is
\[
  \gdigraph_{\mA}(z,w,u) = \frac{\gset(w,z) \hadamard e^{(u-1)
  \eA(z,w)}}{\gset(w,z) \hadamard e^{ {- \eA}(z,w)}}.
\]
\end{theorem}

\begin{proof}
The GGF of the digraph family considered,
where each source-like component is either marked,
or left unmarked by the variable $u$, is $\gdigraph_{\mA}(z,w,u+1)$
(see~\autoref{th:symbolic_method}).
Such a digraph is decomposed as the arrow product
of a set  of SCCs from $\mA$ (the marked source-like components) with a digraph, so
\[
  \gdigraph_{\mA}(z,w,u+1) = \left(\gset(z,w) \hadamard e^{u \eA(z,w)} \right) \gdigraph_{\mA}(z,w).
\]
Taking $u = -1$ gives
\[
  1 = \left(\gset(z,w) \hadamard e^{- \eA(z,w)} \right) \gdigraph_{\mA}(z,w),
  \text{ so } \gdigraph_{\mA}(z,w) = \left(\gset(z,w) \hadamard e^{- \eA(z,w)} \right)^{-1}.
\]
Replacing $u$ with $u-1$ gives
$\gdigraph_{\mA}(z,w,u) = \left(\gset(z,w) \hadamard e^{(u-1) \eA(z,w)} \right) \gdigraph_{\mA}(z,w)$.
\end{proof}

When the family $\mA$ contains only
the SCC with one vertex and no edges, so $\eA(z,w) = z$,
then $\gdigraph_{\mA}(z,w)$ becomes the GGF of DAGs.
Thus, \autoref{th:source_like} generalizes \autoref{th:dag}.
Several interesting corollaries follow.
The first one is our new proof for the EGF of strongly connected digraphs
(original result from~\cite{liskovets1973,liskovets2000some,RobinsonGeneral}).

\begin{corollary} \label{th:strongly_connected}
The exponential generating function of strongly connected digraphs is equal to
\[
  \estrongly(z,w) = - \log \left( \egraph(z,w) \hadamard \frac{1}{\egraph(z,w)}
  \right).
\]
\end{corollary}

\begin{proof}
When $\mA$ is the family of all SCCs,
the first result of \autoref{th:source_like} becomes
\[
  \gdigraph(z,w) = \frac{1}{\gset(w,z) \hadamard e^{- \estrongly(z,w)}}.
\]
By inversion and Hadamard product with $\egraph(z,w)$, we obtain
\[
  e^{- \estrongly(z,w)} = \egraph(z,w) \hadamard \frac{1}{\gdigraph(z,w)}.
\]
Replacing $\gdigraph(z,w)$ with $\egraph(z,w)$
(see \autoref{th:first_egf_ggf})
and taking the logarithm gives the final result.
\end{proof}

This formula enables fast computation of the numbers
of strongly connected digraphs:
$\bigO(n m \log(n + m))$ arithmetic operations
to compute the array of SCCs with at most $n$ vertices and at most $m$ edges,
$\bigO(n \log(n))$ for the SCCs with at most $n$ vertices without edge constraint.
The next corollary might prove useful to investigate
the birth of the giant SCC in random digraph,
following~\cite{JKLP93}.

\begin{corollary}
Consider a nonempty SCC family $\mB$.
The GGF of digraphs with a variable $u$ marking
the number of SCCs from $\mB$ is
\[
  \frac{1}{\gset(w,z) \hadamard e^{(1-u) \eB(z,w) - \estrongly(z,w)}}.
\]
\end{corollary}

\begin{proof}
When $\mA$ is the family of all SCCs,
with an additional variable $u$ marking the SCCs from $\mB$,
then $\eA(z,w,u) = \estrongly(z,w) + (u - 1) \eB(z,w)$,
and the first result of \autoref{th:source_like}
finishes the proof.
\end{proof}

  \subsection{Initially connected digraphs}

\emph{Initially connected digraphs} are defined as
digraphs where any vertex is reachable from the vertex with label \(1\) 
via an oriented path.
Their analysis has been linked to the study of SCCs,
so we provide or recall some results on them for completeness.

\begin{lemma}
For a given number of vertices and edges,
initially connected digraphs with one distinguished vertex are in bijection
with digraphs which have a unique source-like component,
and where one vertex of that component is distinguished.
\end{lemma}

\begin{proof}
Let $\mA$ and $\mB$ denote the two digraph families from the lemma.
Consider a digraph $a \in \mA$.
Since $a$ is initially connected, it contains exactly one source-like SCC.
If the distinguished vertex belongs to the source-like SCC,
then $a \in \mB$.
Otherwise, by switching the distinguished vertex with the vertex of label $1$,
we obtain a digraph from $\mB$.
Reciprocally, if the distinguished vertex of a digraph $b \in \mB$
is in the same SCC as the vertex $1$, then $b \in \mA$.
Otherwise, a digraph from $\mA$ is obtained by switching those two vertices.
\end{proof}

The following lemma provides a relation between
initially connected digraphs and connected graphs
(\cite{liskovets1969initiallyconnected}, proof also available in the conclusion of \cite{JKLP93}).

\begin{lemma} \label{th:initially_connected}
The GGF of initially connected digraphs is equal to the EGF of connected graphs
\[
  \ginitially(z,w) = \econnected(z,w) = \log(\egraph(z,w)).
\]
\end{lemma}

    \section{Conclusion}

Many digraph families can be enumerated using
the same techniques: symbolic method enriched with the arrow product,
Hadamard product,
inclusion-exclusion and additional marking variables.
Marking sinks in DAGs and sink-like SCCs in digraphs can be achieved as well. 
The next challenge is the asymptotics of sparse DAGs, strongly connected graphs,
and, following~\cite{JKLP93}, digraphs phase transition.


\paragraph*{Bibliographic remark.}
The first English version paper we found
containing the elegant expression for the generating function
of strongly connected digraphs
recalled in \autoref{th:strongly_connected} is~\cite{liskovets2000some}.
It points to an earlier publication~\cite{liskovets1973} in Russian,
which contains the proof. Sadly, the authors were not aware of the existence of the
general method described in~\cite{robinson1973,RobinsonGeneral} during the
writing of the current paper.

\paragraph*{Acknowledgements.}
We would like to thank Cyril Banderier
for pointing out the reference \cite{liskovets2000some},
and Vlady Ravelomanana for introducing us to the topic
and for many fruitful discussions.

\end{document}